\documentclass[11pt]{article} 
\usepackage{fullpage}
\usepackage{natbib}

\usepackage{amsthm} 
\usepackage{algorithm} 
\usepackage{algorithmic}

\usepackage{enumitem}
\setlist[itemize]{noitemsep,  leftmargin=*, topsep=0pt}

\usepackage[utf8]{inputenc} 
\usepackage[T1]{fontenc}    
\usepackage{url}            
\usepackage{booktabs}       
\usepackage{amsfonts}       
\usepackage{nicefrac}       
\usepackage{microtype}      
\usepackage{bm}

\usepackage{graphicx} 
\usepackage{subfigure} 
\usepackage{epstopdf}
\usepackage{color}

\usepackage{wrapfig}
\usepackage{booktabs}
\usepackage{amsmath}
\usepackage{amssymb}

\usepackage{breqn}

\newcounter{ass_counter} \newcounter{thm_counter}

\newtheorem{theorem}[thm_counter]{Theorem}
\newtheorem{lemma}[thm_counter]{Lemma}

\newtheorem{assumption}[ass_counter]{Assumption}

\DeclareMathOperator*{\argmin}{arg\,min}

\def\E{\mathbb{E}}

\def\P{\mathcal{P}}
\def\R{\Re}

\def\E{\mathbb{E}}

\def\P{\mathcal{P}}

\newcommand{\beq}{\begin{equation}}
\newcommand{\eeq}{\end{equation}}

\def\1{{\bf{1}}}
\def\0{{\bf{0}}}
\def\I{{\bf{I}}}
\def\w{{\bf w}}

\def\x{{\bf x}}
\def\y{{\bf y}}
\def\z{{\bf z}}

\def\b{{\bf b}}

\def\s{{\bf s}}

\def\prox{\text{prox}}

\def\la{\langle}
\def\ra{\rangle}

\def\Gg{\nabla g}
\def\Gf{\nabla f}

\newcommand{\wb}{\mathbf{w}}

\newcommand{\bphi}{\bm{\phi}}
\newcommand{\bPhi}{\mathbf{\Phi}}

\def\cX{X}
\def\cO{O}
\def\lf{\left}
\def\ri{\right}
\def\SO{{\bf SO}}
\def\sgrad{\nabla}
\def\grad{\nabla}

\allowdisplaybreaks[4]

\title{Accelerating Stochastic Composition Optimization}

\author{
  $^\dag$Mengdi Wang$^*$, $^\ddag$Ji Liu\thanks{Equal contribution.}~, and $^\S$Ethan X. Fang \\
 $^\dag$Princeton University, $^\ddag$University of Rochester, and $^\S$Pennsylvania State University \\
  \texttt{mengdiw@princeton.edu, ji.liu.uwisc@gmail.com, xxf13@psu.edu} \\
}

\begin{document}

\maketitle

\begin{abstract}
Consider the stochastic composition optimization problem where the objective is a composition of two expected-value functions. We propose a new stochastic first-order method, namely the accelerated stochastic compositional proximal gradient (ASC-PG) method, which updates based on queries to the sampling oracle using two different timescales. The ASC-PG is the first proximal gradient method for the stochastic composition problem that can deal with nonsmooth regularization penalty. We show that the ASC-PG exhibits faster convergence than the best known algorithms, and that it achieves the optimal sample-error complexity in several important special cases. We further demonstrate the application of ASC-PG to reinforcement learning  and conduct numerical experiments.
\end{abstract}

\section{Introduction}
\def\cX{X}
\def\cO{O}
\def\lf{\left}
\def\ri{\right}
\def\SO{{\bf SO}}
\def\sgrad{\nabla}
\def\grad{\nabla}

The popular stochastic gradient methods are well suited for minimizing expected-value objective functions or the sum of a large number of loss functions. Stochastic gradient methods find wide applications in estimation, online learning, and training of deep neural networks. Despite their popularity, they do not apply to the minimization of a nonlinear function involving expected values or a composition between two expected-value functions. 

In this paper, we consider the {\it stochastic composition problem}, given by  
\begin{equation}
\min_{\x\in\R^n}\quad H(\x) 
:= \underbrace{\E_v(f_v (\E_w(g_w(\x))))}_{=:F(\x)} + R(\x)
\label{eq:main-problem}
\end{equation}
where $(f\circ g)(\x) = f(g(\x))$ denotes the function composition, $g_w(\cdot):~\R^n \mapsto \R^m$ and $f_v(\cdot):~\R^m\mapsto~\R$ are continuously differentiable functions, $v,w$ are random variables, and $R(\x):~\R^n\mapsto \R\cup \{+\infty\}$ is an extended real-valued closed convex  function.  
We assume throughout that {there exists at least one optimal solution $\x^*$ to problem \eqref{eq:main-problem}}. We focus on the case where $f_v$ and $g_w$ are smooth, but we allow $R$ to be a nonsmooth penalty such as the $\ell_1$-norm. We do no require either the outer function $f_v$ or the inner function $g_w$ to be convex or monotone. The inner and outer random variables $w,v$ can be dependent.

Our algorithmic objective is to develop efficient algorithms for solving problem \eqref{eq:main-problem} based on random evaluations of $f_v$, $g_w$ and their gradients. Our theoretical objective is to analyze the rate of convergence for the stochastic algorithm and to improve it when possible. In the online setting, the iteration complexity of stochastic methods can be interpreted as sample-error complexity of estimating the optimal solution of problem \eqref{eq:main-problem}. 



\subsection{Motivating Examples}

One motivating example is {\it reinforcement learning} \citep{sutton1998reinforcement}. Consider a controllable Markov chain with states $1,\ldots,S$. Estimating the value-per-state of a fixed control policy $\pi$ is known as on-policy learning. It can be casted into an $S\times S$ system of Bellman equations:
$$\gamma P^{\pi} V^{\pi} + r^{\pi} = V^{\pi},$$
where $\gamma \in(0,1)$ is a discount factor, $P_{s\tilde s}^{\pi}$ is the transition probability from state $s$ to state $\tilde s$, and $r_s^{\pi}$ is the expected state transition reward at state $s$. The solution $V^{\pi}$ to the Bellman equation is the value vector, with $V^{\pi}(s)$ being the total expected reward starting at state $s$. In the blackbox simulation environment, $P^{\pi},r^{\pi}$ are unknown but can be sampled from a simulator.
As a result, solving the Bellman equation becomes a special case of the stochastic composition optimization problem:
\begin{align}
\min_{\x\in\Re^S} \quad \|\E[A] \x - \E[\b]\|^2,
\label{eq:rl}
\end{align}
where $A,\b$ are random matrices and random vectors such that $\E[A]=I-\gamma P^{\pi}$ and $\E[\b]=r^{\pi}$. It can be viewed as the composition of the square norm function and the expected linear function.
We will give more details on the reinforcement learning application in Section 4.

Another motivating example is {\it risk-averse learning}. For example, consider the mean-variance minimization problem
$$\min_{\x} \E_{a,b}\quad [ h(\x;a, b)] + {\lambda}\hbox{Var}_{a,b} [h(\x;a, b) ] ,$$
where $h(x;a,b)$ is some loss function parameterized by random variables $a$ and $b$, and $\lambda>0$ is a regularization parameter.
Its batch version takes the form
$$\min_{\x}\quad  \frac1N\sum^N_{i=1} h(\x;a_i, b_i) +  \frac{\lambda}N \sum^N_{i=1}  \lf( h(\x;a_i, b_i) -  \frac1N\sum^N_{i=1} h(\x;a_i, b_i) \ri)^2.$$
Here the variance term is the composition of the mean square function and an expected loss function.
Indeed, the stochastic composition problem (1) finds a broad spectrum of applications in estimation and machine learning. Fast optimization algorithms with theoretical guarantees will lead to new computation tools and online learning methods for a broader problem class, no longer limited to the expectation minimization problem.

\subsection{Related Works and Contributions}

Contrary to the expectation minimization problem,
``unbiased" gradient samples are no longer available for the stochastic composition problem (1). The objective is {\it nonlinear} in the joint probability distribution of $(w,v)$, which substantially complicates the problem. In a recent work by \cite{dentcheva2015statistical}, a special case of the stochastic composition problem, i.e., risk-averse optimization, has been studied. A central limit theorem has been established, showing that the $K$-sample batch problem converges to the true problem at the rate of $\cO(1/\sqrt{K})$ in a proper sense. 
For the case where $R(x)=0$, \cite{wang2014stochastic} has proposed and analyzed a class of {\it stochastic compositional gradient/subgradient methods} (SCGD). The SCGD involves two iterations of different time scales, one for estimating $x^*$ by a stochastic quasi-gradient iteration, the other for maintaining a running estimate of $g(x^*)$. Almost sure convergence and several convergence rate results have been obtained.

The idea of using two-timescale quasi-gradient traced back to the earlier work \cite{Erm76}.  
The incremental treatment of proximal gradient iteration has been studied extensively for the expectation minimization problem, see for examples \cite{NeB01,Ber11a,Ned11,WaB13,beck2009fast,gurbuzbalaban2015convergence,RSS11,ghadimi2015accelerated,ShZ12}.
However, except for \cite{wang2014stochastic}, all of these works focus on the expectation minimization problem and do not apply to the stochastic composition problem \eqref{eq:main-problem}.

In this paper, we propose a new accelerated stochastic compositional proximal gradient (ASC-PG) method that applies to the more general penalized problem \eqref{eq:main-problem}. We use a coupled martingale stochastic analysis to show that ASC-PG achieves significantly better sample-error complexity in various cases. We also show that ASC-PG exhibits optimal sample-error complexity in two important special cases: the case where the outer function is linear and the case where the inner function is linear.

Our contributions are summarized as follows:
\begin{enumerate} [noitemsep,topsep=0pt, leftmargin=*]
\item We propose the first stochastic {\it proximal-gradient} method for the stochastic composition problem. This is the first algorithm that is able to address the nonsmooth regularization penalty $R(\cdot)$ without deteriorating the convergence rate.
\item We obtain a convergence rate $O(K^{-4/9})$ for smooth optimization problems that are not necessarily convex, where $K$ is the number of queries to the stochastic first-order oracle. This improves the best known convergence rate and provides a new benchmark for the stochastic composition~problem.
\item We provide a comprehensive analysis and results that apply to various special cases. In particular, our results contain as special cases the known optimal rate results for the expectation minimization problem, i.e., $\cO(1/\sqrt{K})$ for general objectives and $\cO(1/K)$ for strongly convex objectives. 
\item In the special case where the inner function $g(\cdot)$ is a linear mapping, we show that it is sufficient to use one timescale to guarantee convergence. Our result achieves the non-improvable rate of convergence $\cO(1/{K})$. It implies that the inner linearity does not bring fundamental difficulty to the stochastic composition problem.
\item We show that the proposed method leads to a new on-policy reinforcement learning algorithm. The new learning algorithm achieves the optimal convergence rate $\cO(1/\sqrt{K})$ for solving Bellman equations based on $K$ observed state transitions.
\end{enumerate}


In comparison with \cite{wang2014stochastic}, our analysis is more succinct and leads to stronger results. To the best of our knowledge, results in this paper provide the best-known rates for the stochastic composition~problem. 

\paragraph{Paper Organization.} Section 2 states the sampling oracle and the accelerated stochastic compositional proximal gradient algorithm (ASC-PG). Section 3 states the convergence rate results in the case of  general nonconvex objective and in the case of  strongly convex objective, respectively. Section 4 describes an application of ASC-PG to reinforcement learning and gives numerical experiments. 

\paragraph{Notations  and Definitions.}
For $\x\in\Re^n$, we denote by $\x'$ its transpose, and by $\|\x\|$ its Euclidean norm (i.e., $\|\x\|=\sqrt{\x'\x}$). 
For two sequences $\{\y_k\}$ and $\{\z_k\}$, we write $\y_k=\cO(\z_k)$ if there exists a constant $c>0$ such that $\|\y_k\|\leq c\|\z_k\|$ for each $k$. 
We denote by $\I_{\text{condition}}^{\text{value}}$ the indicator function, which returns ``value'' if the ``condition'' is satisfied; otherwise $0$.
We denote by $H^*$ the optimal objective function value for \eqref{eq:main-problem}, denote by $X^*$ the set of optimal solutions, and denote by $\P_S(\x)$ the Euclidean projection of $\x$ onto $S$ for any convex set $S$. We let $f(\y) = \E_{v}[f_v(\y)]$ and $g(\x) = \E_w[g_w(\x)]$.

\section{Algorithm}

We focus on the blackbox sampling environment. Suppose that we have access to a stochastic first-order oracle, which returns random realizations of first-order information upon queries. This is a typical simulation oracle that is available in both online and batch learning.
More specifically, assume that we are given a {\bf Sampling Oracle} (\SO) such that
\begin{itemize}
\item Given some $\x\in \R^n$, the \SO\ returns a random vector 
$g_w(x)$ and a noisy subgradient 
$\sgrad g_w (\x).$
\item Given some $\y\in\R^m$, the \SO\ returns a noisy gradient $\grad f_v(\y).$
\end{itemize}
Now we propose the Accelerated Stochastic Compositional Proximal Gradient (ASC-PG) algorithm, see Algorithm 1. ASC-PG is a generalization of the SCGD proposed by \citet{wang2014stochastic}, in which a proximal step is used to replace the projection step. 

\begin{algorithm}
\caption{Accelerated Stochastic Compositional Proximal Gradient (ASC-PG)}\label{alg}
\begin{algorithmic}[1]
\REQUIRE $\x_1 \in \R^{n}$, $\y_0 \in \R^m$, \SO, $K$, stepsize sequences $\{\alpha_k\}_{k=1}^K$, and $\{\beta_k\}_{k=1}^K$.
\ENSURE $\{\x_k\}_{k=1}^K$
\FOR{$k=1,\cdots, K$}
\STATE Query the \SO\ and obtain gradient samples $\nabla f_{v_k}(\y_k) $, $\nabla g_{w_k}(\z_k) .$
\STATE Update the main iterate by
\begin{eqnarray*}
\x_{k+1} & = & \prox_{\alpha_k R(\cdot)}\left(\x_k - \alpha_k \nabla g_{w_k}^\top(\x_k) \nabla f_{v_k}(\y_k)\right).
\end{eqnarray*}
\STATE Update auxillary iterates by an {\it extrapolation-smoothing} scheme:
\begin{eqnarray*}
\z_{k+1} &=& \left(1-{1\over \beta_k}\right) \x_k + {1\over \beta_k} \x_{k+1},
\\
\y_{k+1} & = & (1 - \beta_{k})\y_{k} + \beta_{k} g_{w_{k+1}}(\z_{k+1}) ,
\end{eqnarray*}
where the sample $g_{w_{k+1}}(\z_{k+1})$ is obtained via querying the \SO.
\ENDFOR
\end{algorithmic}
\end{algorithm}

In Algorithm 1, the {\it extrapolation-smoothing scheme} (i.e., the $(\y,\z)$-step) is critical for convergence acceleration. The acceleration is due to the fast running estimation of the unknown quantity $g(\x_k):=\E_{w}[g_w(\x_k)]$. At iteration $k$, the running estimate $\y_k$ of $g(\x_k)$ is obtained using a weighted smoothing scheme, corresponding to the $\y$-step; while the new query point $\z_{k+1}$ is obtained through extrapolation, corresponding to the $\z$-step. The updates are constructed in a way such that $\y_k$ is a nearly unbiased estimate of $g(\x_k).$  
To see how the {extrapolation-smoothing scheme} works, we define the weights as
\begin{equation}
\xi_t^{(k)} = 
\begin{cases}
    \beta_t \prod_{i=t+1}^k(1-\beta_i),& \text{if } k > t \geq 0\\
    \beta_k,              & \text{if } k=t\geq 0.
\end{cases}
\end{equation}
Then we can verify the following important relations:
\[
\x_{k+1} = \sum_{t=0}^k \xi_t^{(k)} \z_{t+1},\qquad \y_{k+1} = \sum_{t=0}^k \xi_t^{(k)}  g_{w_{t+1}}(\z_{t+1}).
\]
Now consider the special case where $g_w(\cdot)$ is always a linear mapping $g_w(z) = A_w z+b_z$ and $\beta_k=\xi_t^{(k)} =1/(k+1)$. Then we have 
\[
g(\x_{k+1}) = \frac1{k+1} \sum_{t=0}^k \E[A_w]  \z_{t+1} +\E[\b_w],
\qquad \y_{k+1} = \frac1{k+1}\sum_{t=0}^k   A_{w_{t+1}}\z_{t+1}+ \frac1{k+1}\sum_{t=0}^k   \b_{w_{t+1}}.
\]
In this way, we can see that the scaled error $$k(\y_{k+1} - g(\x_{k+1})) =   \sum_{t=0}^k (A_{w_{t+1}}- \E[A_w] )\z_{t+1}+ \sum_{t=0}^k    (\b_{w_{t+1}} - \E[\b_w] )$$ 
is a {\it zero-mean} and {\it zero-drift} martingale. Under additional technical assumptions, we have
$$\E[\|\y_{k+1} - g(\x_{k+1}) \|^2] \leq \cO\lf(\frac1k\ri).$$
Note that the zero-drift property of the error martingale is the key to the fast convergence rate. The zero-drift property comes from the near-unbiasedness of $\y_k$, which is due to the special construction of the extrapolation-smoothing scheme.
In the more general case where $g_w$ is not necessarily linear, we can use a similar argument to show that $\y_k$ is a nearly unbiased estimate of $g(\x_k)$. As a result, the extrapolation-smoothing ($\y,\z$)-step ensures that $\y_k$ tracks the unknown quantity $g(\x_k)$ efficiently.

\section{Main Results}
We present our main theoretical results in this section. Let us begin by stating our assumptions. Note that all assumptions involving random realizations of $v,w$ hold with probability 1.
\begin{assumption} \label{ass:unb}
The samples generated by the \SO\ are unbiased in the following sense:
\begin{enumerate}
\item $\E_{\{w_k, v_k\}}(\nabla g^\top_{w_k}(\x) \nabla f_{v_k}(\y)) = \nabla g^\top(\x) \nabla f(\y)\quad \forall k=1,2,\cdots, K,\quad \forall \x,\forall \y$.
\item $\E_{w_k}(g_{w_k}(\x)) = g(\x)\quad \forall \x$.
\end{enumerate}
Note that $w_k$ and $v_k$ are not necessarily independent.
\end{assumption}

\begin{assumption} \label{ass:var}
The sample gradients and values generated by the \SO\ satisfy
$$\E_w(\|g_w(\x) - g(\x)\|^2) \leq \sigma^2\quad \forall \x.$$

\end{assumption}

\begin{assumption} \label{ass:sg}
The sample gradients generated by the \SO\ are uniformly bounded, and the penalty function $R$ has bounded gradients.
$$\|\nabla f_v(\x)\| \leq \Theta(1) , \|\nabla g_w(\x)\| \leq \Theta(1) , \|\partial R(\x)\| \leq \Theta(1)  \quad \forall \x, \forall w, \forall v $$
\end{assumption}

\begin{assumption} \label{ass:lp}
There exists $L_F,L_f,L_g>0$ such that
the inner and outer functions satisfying the following Lipschitzian conditions
\begin{enumerate}
\item $F(\z)-F(\x) \leq \la \nabla F(\x), \z- \x \ra + {L_F\over 2} \|\z- \x\|^2 \quad \forall \x~\forall \z$.
\item $\|\nabla f_v(\y) - \nabla f_v(\w)\| \leq L_f\|\y - \w\|\quad \forall \y~\forall \w~\forall v.$
\item $\|g(\x) - g(\z) - \nabla g(\z)^\top (\x -\z)\| \leq {L_g\over 2} \|\x - \z\|^2\quad \forall \x~\forall \z.$
\end{enumerate}
\end{assumption}

Our first main result concerns with general optimization problems which are not necessarily convex. 

\begin{theorem}[Smooth Optimization] \label{thm:nonconvex}
Let
Assumptions~\ref{ass:unb}, \ref{ass:var}, \ref{ass:sg}, and \ref{ass:lp} hold.
Denote by $F(\x) := (\E_v(f_v) \circ \E_w(g_w))(\x)$ for short and suppose that $R(\x)=0$ in \eqref{eq:main-problem} and $\E(F(\x_k))$ is bounded from above. Choose $\alpha_k = k^{-a}$ and $\beta_k = 2k^{-b}$ where $a\in (0,1)$ and $b\in (0,1)$ in Algorithm~\ref{alg}. Then we have
\begin{align}
\frac{\sum_{k=1}^K \E(\|\nabla F(\x_k)\|^2)}{K} \leq 
O(K^{a-1} + L_f^2 L_gK^{4b-4a}\I_{4a-4b=1}^{\log K} + L_f^2  K^{-b}+ K^{-a}).
\label{eq:thm:non:1}
\end{align}
If $L_g\neq 0$ and $L_f\neq 0$, choose $a=5/9$ and $b=4/9$, yielding
\begin{align}
\frac{\sum_{k=1}^K \E(\|\nabla F(\x_k)\|^2)}{K} \leq 
O(K^{-4/9}).
\label{eq:thm:non:2}
\end{align}
If $L_g=0$ or $L_f=0$, then the optimal $a$ and $b$ can be chosen to be $a=b=1/2$, yielding
\begin{align}
\frac{\sum_{k=1}^K \E(\|\nabla F(\x_k)\|^2)}{K} \leq O(K^{-1/2}).
\label{eq:thm:non:3}
\end{align}
\end{theorem}

The result of Theorem 1 strictly improves the corresponding results in \citet{wang2014stochastic}. First the result in \eqref{eq:thm:non:2} improves the convergence rate from $O(k^{-2/7})$ to $O(k^{-4/9})$ for the general case. This improves the best known convergence rate and provides a new benchmark for the stochastic composition problem.  

Our second main result concerns strongly convex objective functions. We say that the objective function $H$ is {\it optimally strongly convex} with parameter $\lambda>0$ if 
\begin{eqnarray}
H(\x) - H(\P_{X^*}(\x)) \geq \lambda \|\x-\P_{X^*}(\x)\|^2\quad \forall \x.
\label{eq:osc}
\end{eqnarray}
(see \citet{liu2015asynchronous}).
Note that any strongly convex function is optimally strongly convex, but the reverse does not hold. For example, the objective function \eqref{eq:rl} in on-policy reinforcement learning is always optimally strongly convex (even if $\E(A)$ is a rank deficient matrix), but not necessarily strongly convex. 

\begin{theorem}\label{thm:strconvex} (Strongly Convex Optimization)
Suppose that the objective function $H(\x)$ in \eqref{eq:main-problem} is optimally strongly convex with parameter $\lambda>0$ defined in \eqref{eq:osc}. Set $\alpha_k = C_ak^{-a}$ and $\beta_k = C_b k^{-b}$ where $C_a>4\lambda$, $C_b>2$, $a\in (0,1]$, and $b\in (0, 1]$ in Algorithm~\ref{alg}. Under Assumptions~\ref{ass:unb}, \ref{ass:var}, \ref{ass:sg}, and \ref{ass:lp}, we have 
\begin{equation}
\label{eq:thm:osc:1}
\E(\|\x_k -\P_{X^*}(\x_k)\|^2) \leq O\left(k^{-a} + L_f^2 L_gk^{-4a+4b} + L_f^2 k^{-b}\right).
\end{equation}
If $L_g\neq 0$ and $L_f\neq 0$, choose $a=1$ and $b=4/5$, yielding 
\begin{align}
\E(\|\x_k -\P_{X^*}(\x_k)\|^2) \leq O(k^{-4/5}).
\label{eq:thm:osc:2}
\end{align}
If $L_g=0$ or $L_f=0$, choose $a=1$ and $b=1$, yielding 
\begin{align}
\E(\|\x_k -\P_{X^*}(\x_k)\|^2) \leq O(k^{-1}).
\label{eq:thm:osc:3}
\end{align}
\end{theorem}

Let us discuss the results of Theorem 2. In the general case where $L_f\neq 0$ and $L_g\neq 0$, the convergence rate in \eqref{eq:thm:osc:2} is consistent with the result of \citet{wang2014stochastic}. 
Now consider the special case where $L_g=0$, i.e., the inner mapping is linear. This result finds an immediate application to Bellman error minimization problem \eqref{eq:rl} which arises from reinforcement learning problem in (and with $\ell_1$ norm regularization). The proposed ASC-PG algorithm is able to achieve the optimal rate $O(1/K)$ without any assumption on $f_v$. To the best of our knowledge, this is the best (also optimal) sample-error complexity for on-policy reinforcement learning.   

\paragraph{Remarks} Theorems 1 and 2 give important implications about the special cases where $L_f = 0$ or $L_g=0$. In these cases, we argue that our convergence rate \eqref{eq:thm:osc:3} is ``{\it optimal}" with respect to the sample size $k$. To see this, it is worth pointing out the the $\cO(1/K)$ rate of convergence is optimal for strongly convex expectation minimization problem. Because the expectation minimization problem is a special case of problem \eqref{eq:main-problem}, the $\cO(1/K)$ convergence rate must be optimal for the stochastic composition problem too.  
\begin{itemize} [noitemsep,topsep=0pt]
\item
Consider the case where $L_f=0$, which means that the outer function $f_v(\cdot)$ is linear with probability 1. Then the stochastic composition problem \eqref{eq:main-problem} reduces to an expectation minimization problem since $(\E_v f_v \circ \E_wg_w)(\x) = \E_v (f_v (\E_w g_w(\x))) = \E_v \E_w (f_v\circ g_w)(\x)$. Therefore, it makes a perfect sense to obtain the optimal convergence rate. 
\item
Consider the case where $L_g=0$, which means that the inner function $g(\cdot)$ is a linear mapping. The result is quite surprising.  Note that even $g(\cdot)$ is a linear mapping, it does not reduce problem \eqref{eq:main-problem} to an expectation minimization problem. However, the ASC-PG still achieves the optimal convergence rate. This suggests that, when inner linearity holds, the stochastic composition problem \eqref{eq:main-problem} is not fundamentally more difficult than the expectation minimization problem. 
\end{itemize}
The convergence rate results unveiled in Theorems 1 and 2 are the best  known results for the composition problem. We believe that they provide important new result which provides insights into the complexity of the stochastic composition problem.

\section{Application to Reinforcement Learning}
  \begin{figure}[ht!]
\centering
\subfigure{
\includegraphics[width = 0.38\textwidth]{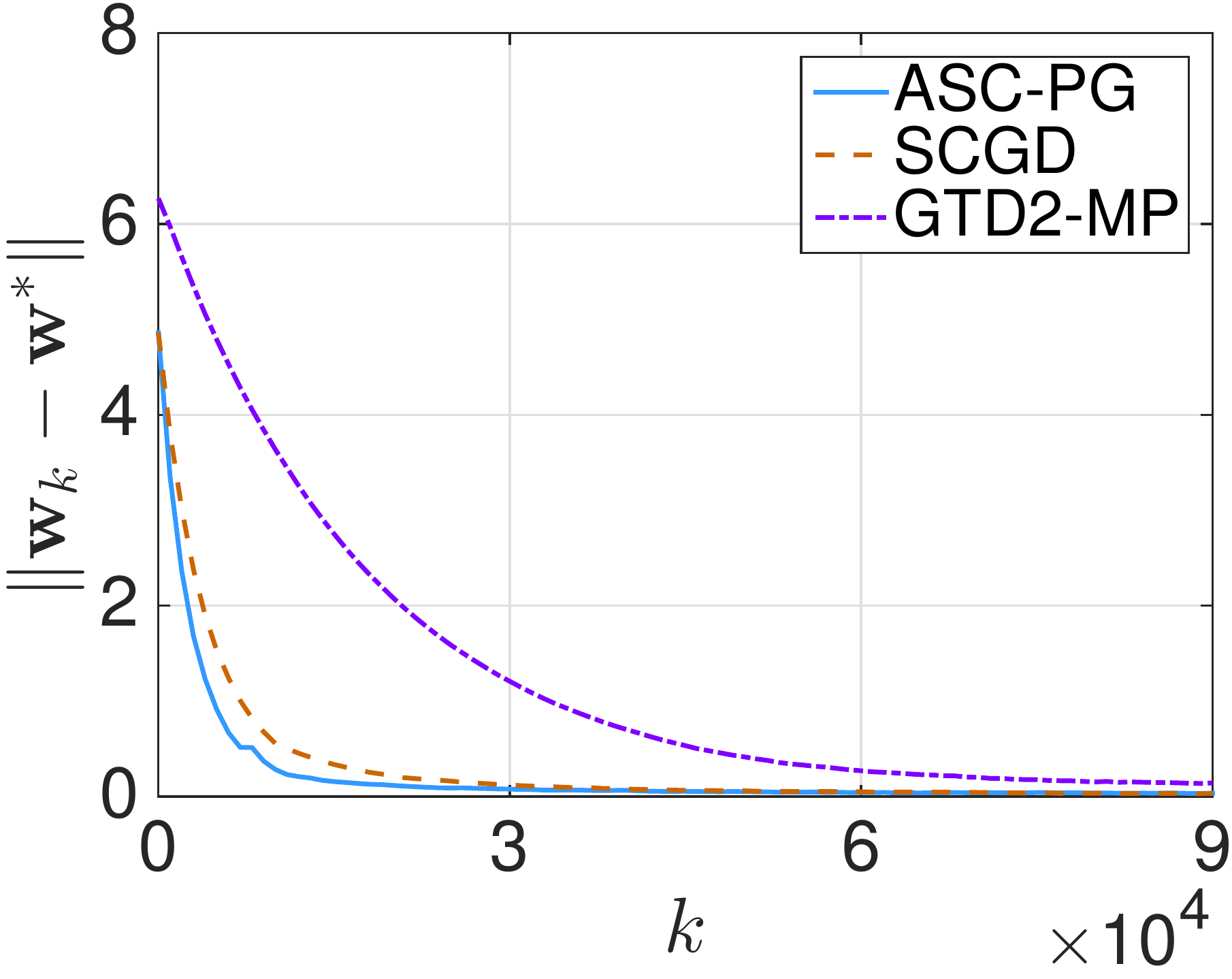}
}
\subfigure{
\includegraphics[width = 0.38\textwidth]{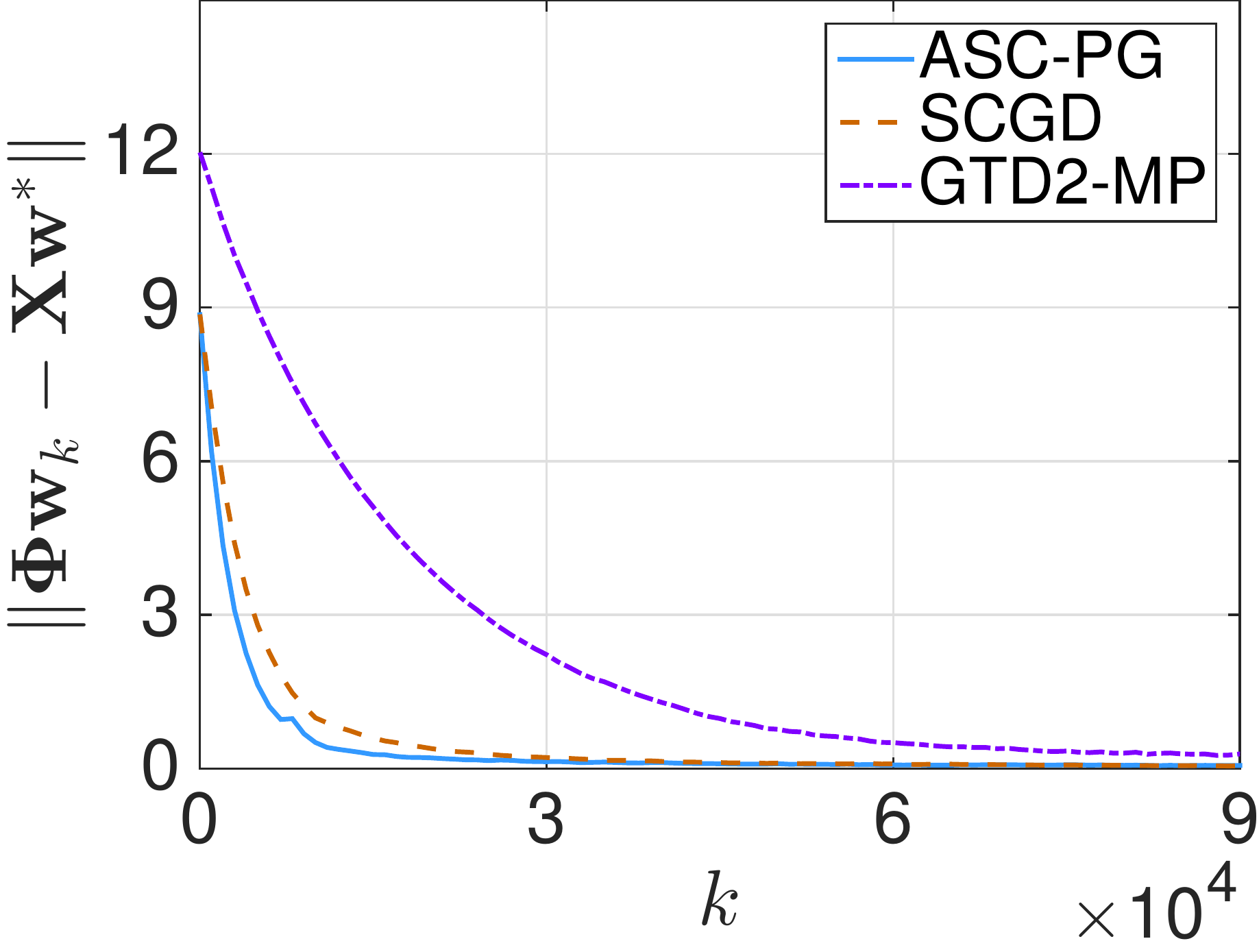}
}\\
\subfigure{
\includegraphics[width = 0.38\textwidth]{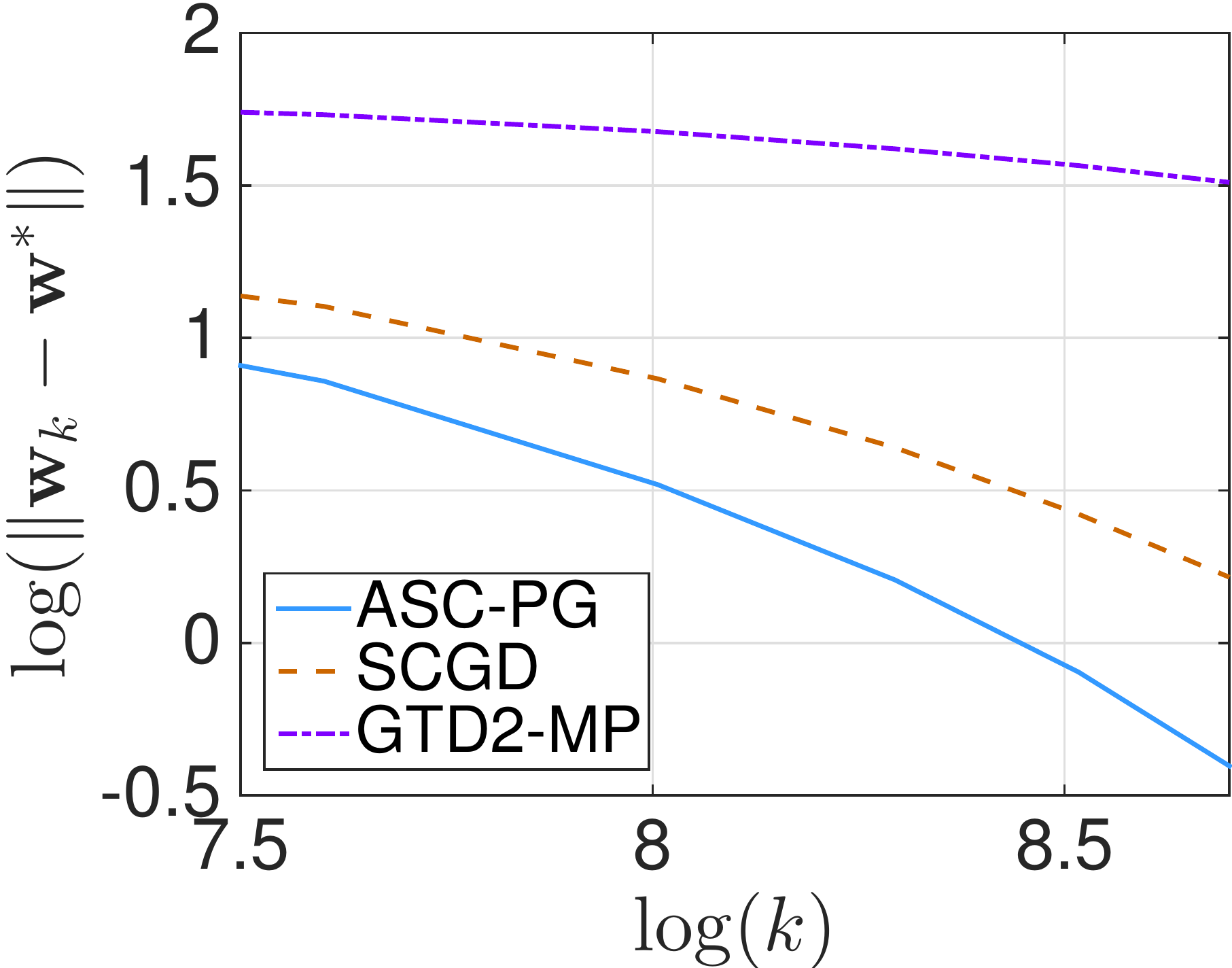}
}
\subfigure{
\includegraphics[width = 0.38\textwidth]{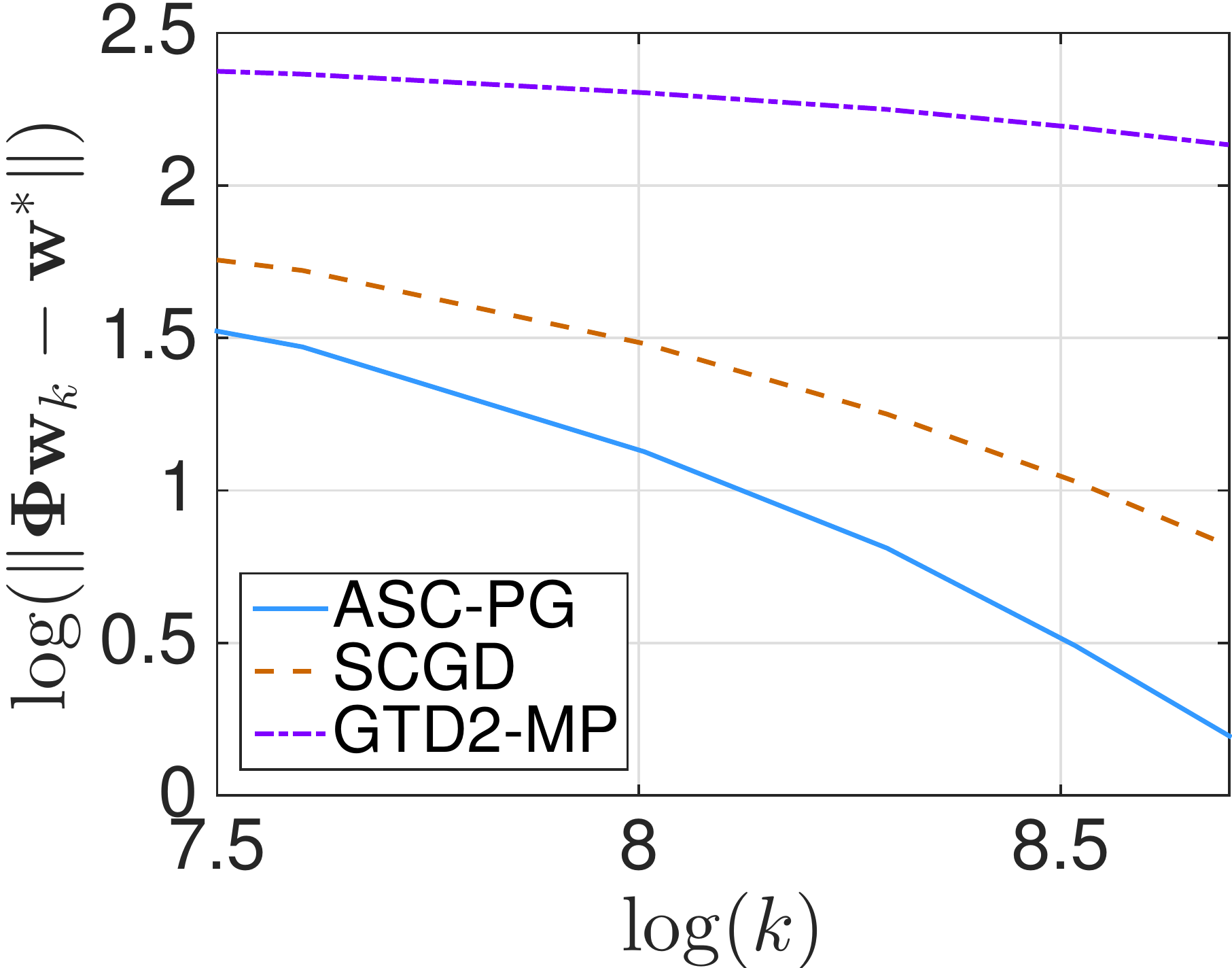}
}
\caption{Empirical convergence rate of the ASC-PG algorithm and the GTD2-MP algorithm under Experiment 1 averaged over 100 runs, where $\wb_k$ denotes the solution at the $k$-th iteration.}
\vspace{-3mm}
\label{fig:baird}
\end{figure}

In this section, we apply the proposed ASC-PG algorithm to conduct policy value evaluation in 
reinforcement learning through attacking Bellman equations. Suppose that there are in total $S$ states. Let the policy of interest be $\pi$. Denote the value function of states by $V^\pi\in\R^S$, where $V^\pi(s)$ denotes the value of being at state $s$ under policy $\pi$. The Bellman equation of the problem is
$$
V^{\pi}(s_1) = \E_\pi\big\{ r_{s_1,s_2} + \gamma \cdot V^{\pi}(s_2)\big| s_1\big\} \text{ for all }s_1,s_2\in\{1,...,S\},
$$
where $r_{s_1,s_2}$ denotes the reward of moving from state $s_1$ to $s_2$, and  the expectation is taken over all possible future state $s_2$ conditioned on current state $s_1$ and the policy $\pi$. We have that the solution $V^*\in\R^S$ to the above equation satisfies that $V^* = V^\pi$.  Here a moderately large $S$ will make solving the Bellman equation directly impractical. To resolve the curse of dimensionality, in many practical applications, we approximate the value of each state by some linear map of its feature $\bphi_s\in \R^m$, where $d < S$ to reduce the dimension. In particular, we assume that $V^\pi(s) \approx \bphi_s^T\wb^*$  for some $\wb^*\in \R^m$. 

To compute $\wb^*$, we formulate the problem as a Bellman residual minimization problem that 
$$
\min_{\wb} \sum_{s=1}^S(\bphi_s^T\wb - q_{\pi,s'}(\wb) )^2, 
$$
where 
$q_{\pi,s'}(\wb) = \E_\pi\big\{r_{s,s'} +  \gamma\cdot \bphi_{s'}\wb   \big\} = \sum_{s'}P^{\pi}_{ss'} (\{r_{s,s'} +  \gamma\cdot \bphi_{s'}\wb)$; $\gamma < 1$ is a discount factor, and $r_{s,s'}$ is the random reward of transition from state $s$ to state $s'$. It is clearly seen that the proposed ASC-PG algorithm could be directly applied to solve this problem where we take 
\begin{gather*}
g(\wb) = \big(\bphi_1^T\wb,q_{\pi,1}(\wb),...,\bphi_S^T\wb,q_{\pi,S}(\wb)\big) \in \R^{2S},\\
f\Big( \big(\bphi_1^T\wb,q_{\pi,1}(\wb),...,\bphi_S^T\wb,q_{\pi,S}(\wb)\big)\Big) =  \sum_{s=1}^S(\bphi_s\wb - q_{\pi,s'}(\wb) )^2\in \R.
\end{gather*}
We point out that  the $g(\cdot)$ function here is a linear map. By our theoretical analysis, we expect to achieve a faster $\cO(1/k)$ rate of convergence, which is justified empirically in our later simulation study.

We consider three experiments, where in the first two experiments, 
we compare our proposed accelerated ASC-PG algorithm with SCGD algorithm \citep{wang2014stochastic} and the recently proposed GTD2-MP algorithm \citep{liu2015finite}. Also, in the first two experiments, we do not add any regularization term, i.e. $R(\cdot) = 0$. In the third experiment, we add an $\ell_1$-penalization term $\lambda\|\wb\|_1$. In all cases, we choose the step sizes via comparison studies as in \cite{dann2014policy}: 

\begin{itemize}
\item Experiment 1: We use the Baird's example \citep{baird1995residual}, which is a well-known example to test the off-policy convergent algorithms. This example contains $S=6$ states, and two actions at each state. We refer the readers to \cite{baird1995residual} for more detailed information of the example. 
\item Experiment 2: We  generate a Markov decision problem (MDP) using similar setup as in \cite{white2016investigating}. In each instance, we randomly generate an MDP  which contains $S = 100$ states, and three actions at each state. The dimension of the  Given one state and one action, the agent can move to one of four next possible states. In our simulation, we generate the transition probabilities for each MDP instance uniformly from $[0,1]$ and normalize the sum of transitions to one, and we generate the reward for each transition also uniformly in $[0,1]$. 
\item Experiment 3: We generate the data same as Experiment 2 except that we have a larger $d = 100$ dimensional feature space, where only the first $4$ components of $\wb^*$ are non-zeros. 
We add an $\ell_1$-regularization term, $\lambda\|\wb\|_1$, to the objective function. 
\end{itemize}


Denote by $\wb_k$ the solution at the $k$-th iteration. 
For the first two experiments, we report the empirical convergence performance $\|\wb_k - \wb^*\|$ and $\|\bPhi\wb_k - \bPhi\wb^*\|$, where $\bPhi = (\bphi_1,...,\bphi_S)^T\in\R^{S\times d}$ and $\bPhi\wb^* = V$, and all $\wb_k$'s  are averaged over 100 runs,  in the first two subfigures of Figures \ref{fig:baird} and \ref{fig:mdp}. It is seen that  the ASC-PG algorithm achieves the fastest convergence rate empirically in both experiments.  
To further evaluate our theoretical results, we plot $\log(t)$ vs. $\log(\|\wb_k - \wb^*\|)$ (or $\log(\|\bPhi\wb_k - \bPhi^*\|)$ averaged over 100 runs  for the first two experiments in the second two subfigures of Figures \ref{fig:baird} and \ref{fig:mdp}. The empirical results further support our theoretical analysis that $\|\wb_k - \wb^*\|^2 = \cO(1/k)$ for the ASC-PG algorithm when $g(\cdot)$ is a linear mapping.

For Experiment 3, as the optimal solution is unknown, we run the ASC-PG algorithm for one million iterations and take the corresponding solution as the optimal solution $\hat\wb^*$, and we report $\|\wb_k - \hat\wb^*\|$ and $\|\bPhi\wb_k - \bPhi\hat\wb^*\|$ averaged over 100 runs in Figure \ref{fig:l1}. It is seen the the ASC-PG algorithm achieves fast empirical convergence rate.


  \begin{figure}[ht!]
\centering
\subfigure{
\includegraphics[width = 0.38\textwidth]{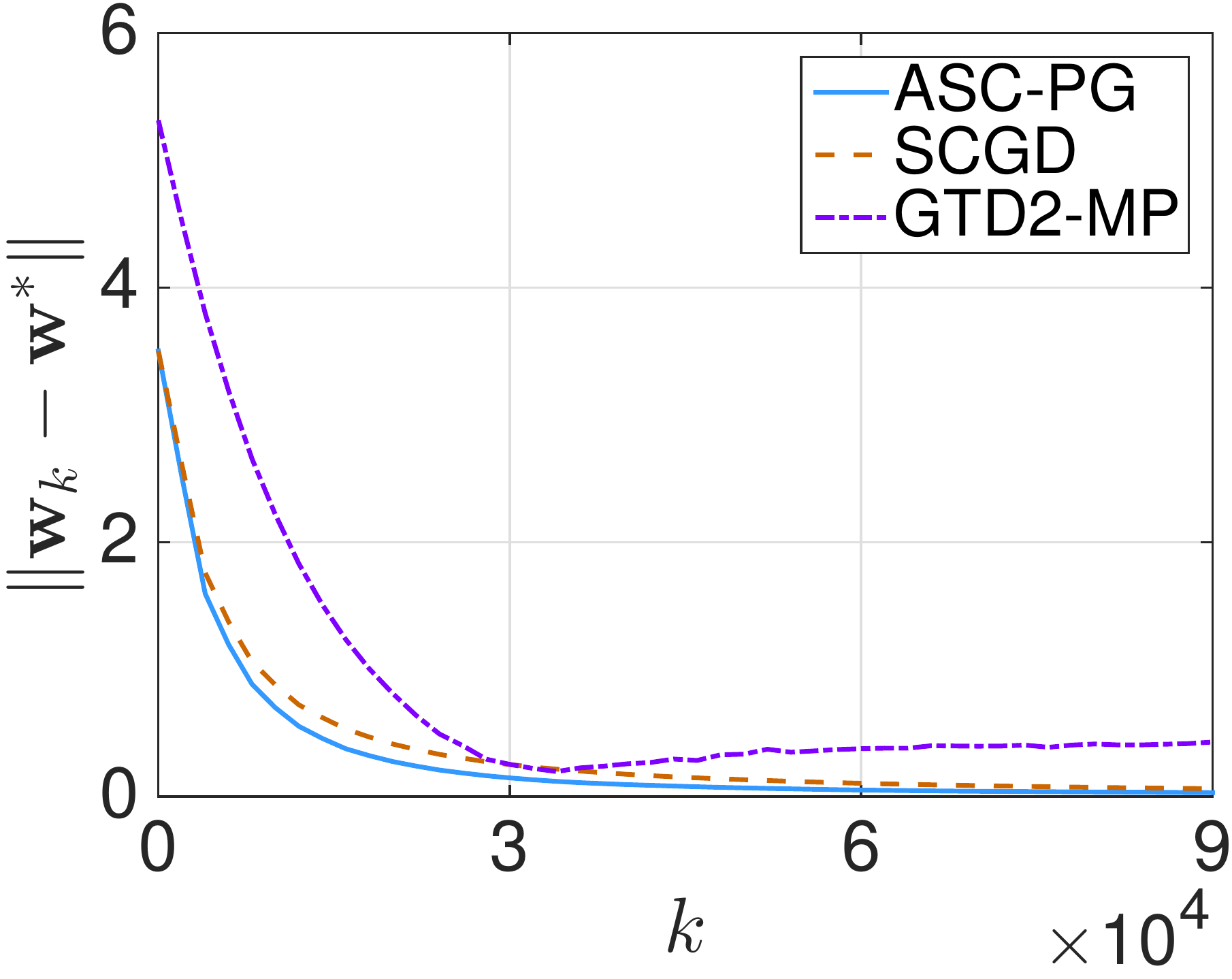}
}
\subfigure{
\includegraphics[width = 0.38\textwidth]{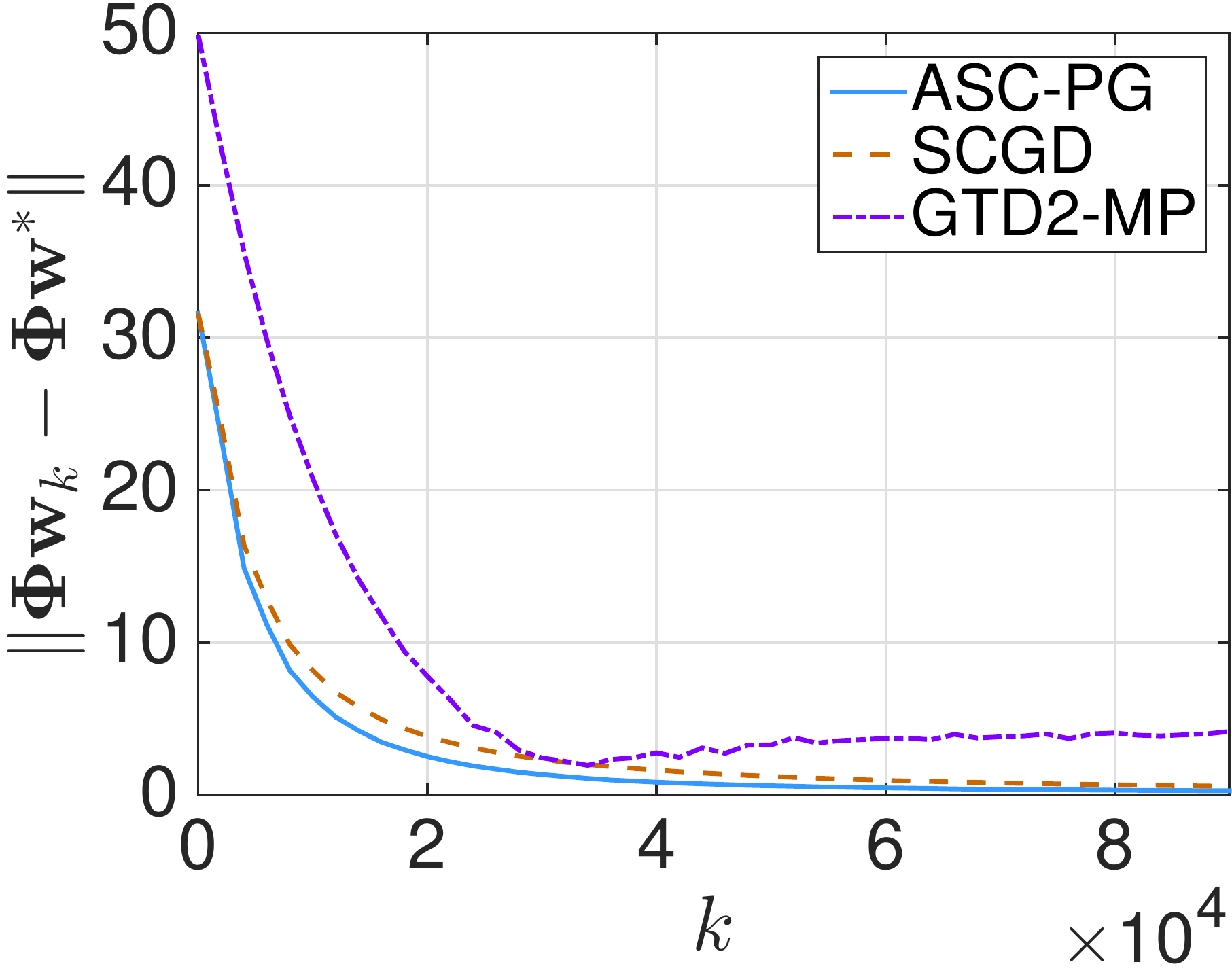}
}\\
\subfigure{
\includegraphics[width = 0.38\textwidth]{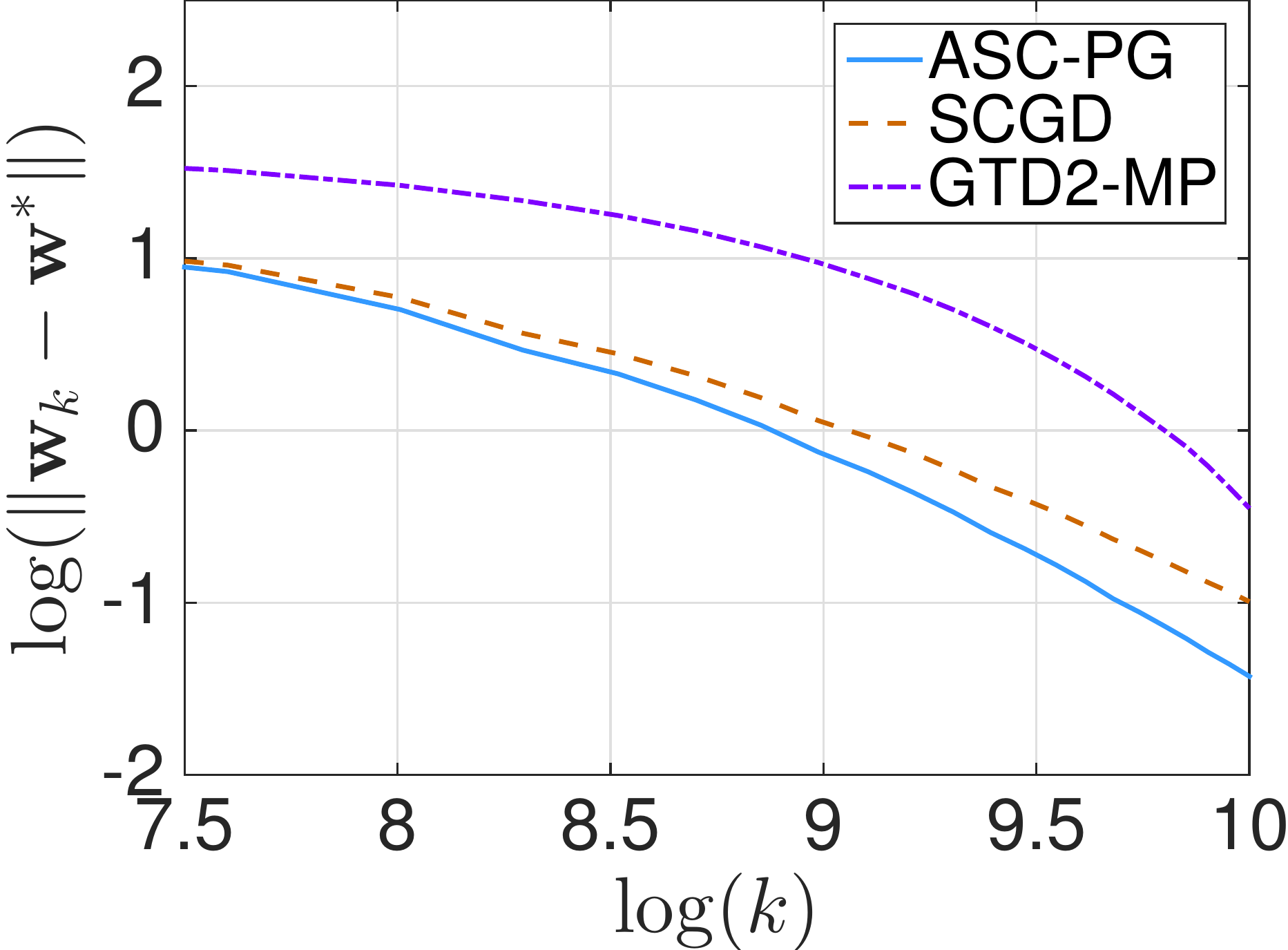}
}
\subfigure{
\includegraphics[width = 0.38\textwidth]{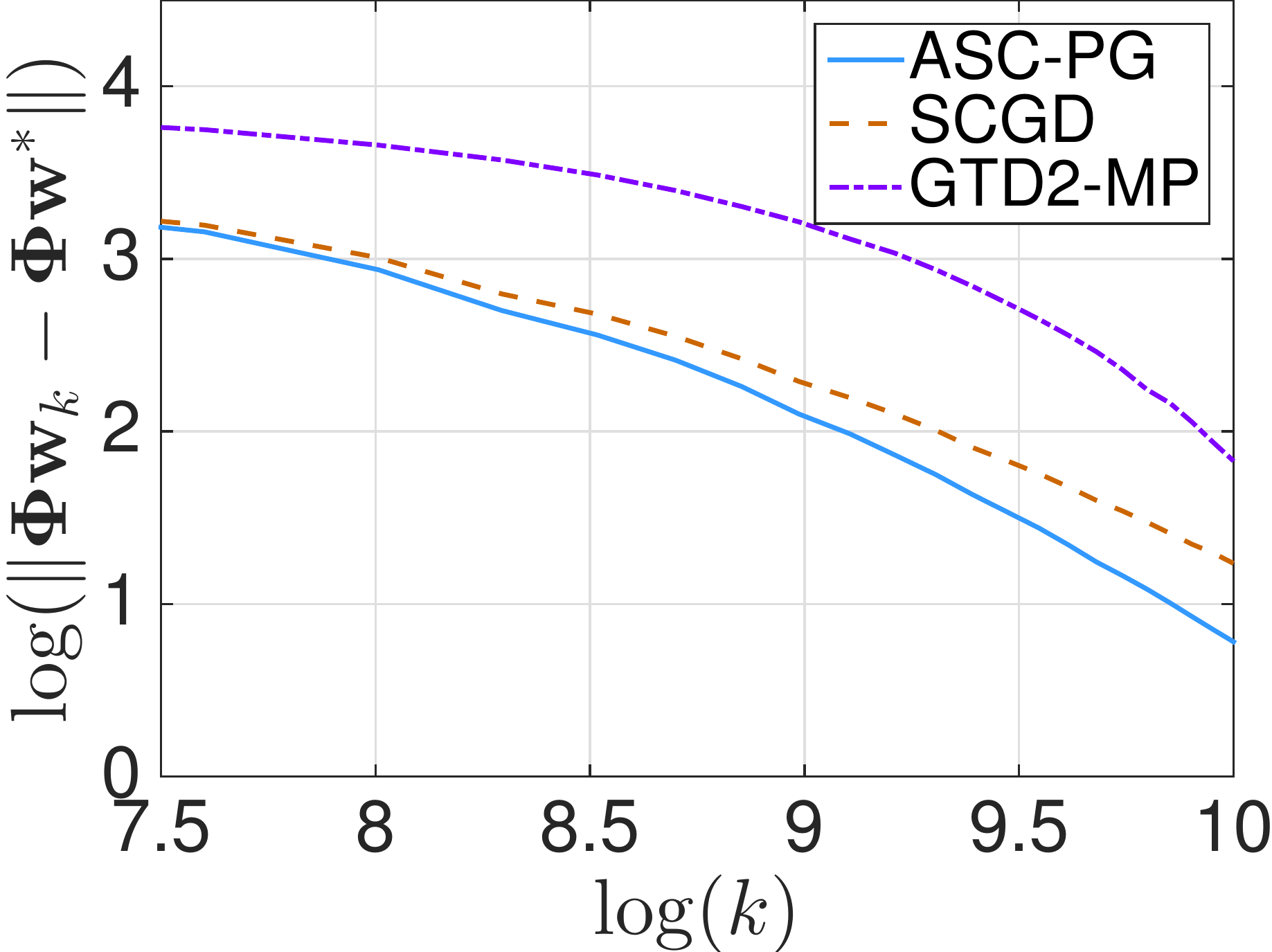}
}
\caption{Empirical convergence rate of the ASC-PG algorithm and the GTD2-MP algorithm under Experiment 2 averaged over 100 runs, where $\wb_k$ denotes the solution at the $k$-th iteration.}
\label{fig:mdp}
\end{figure}

  \begin{figure}[ht!]
\centering
\subfigure{
\includegraphics[width = 0.38\textwidth]{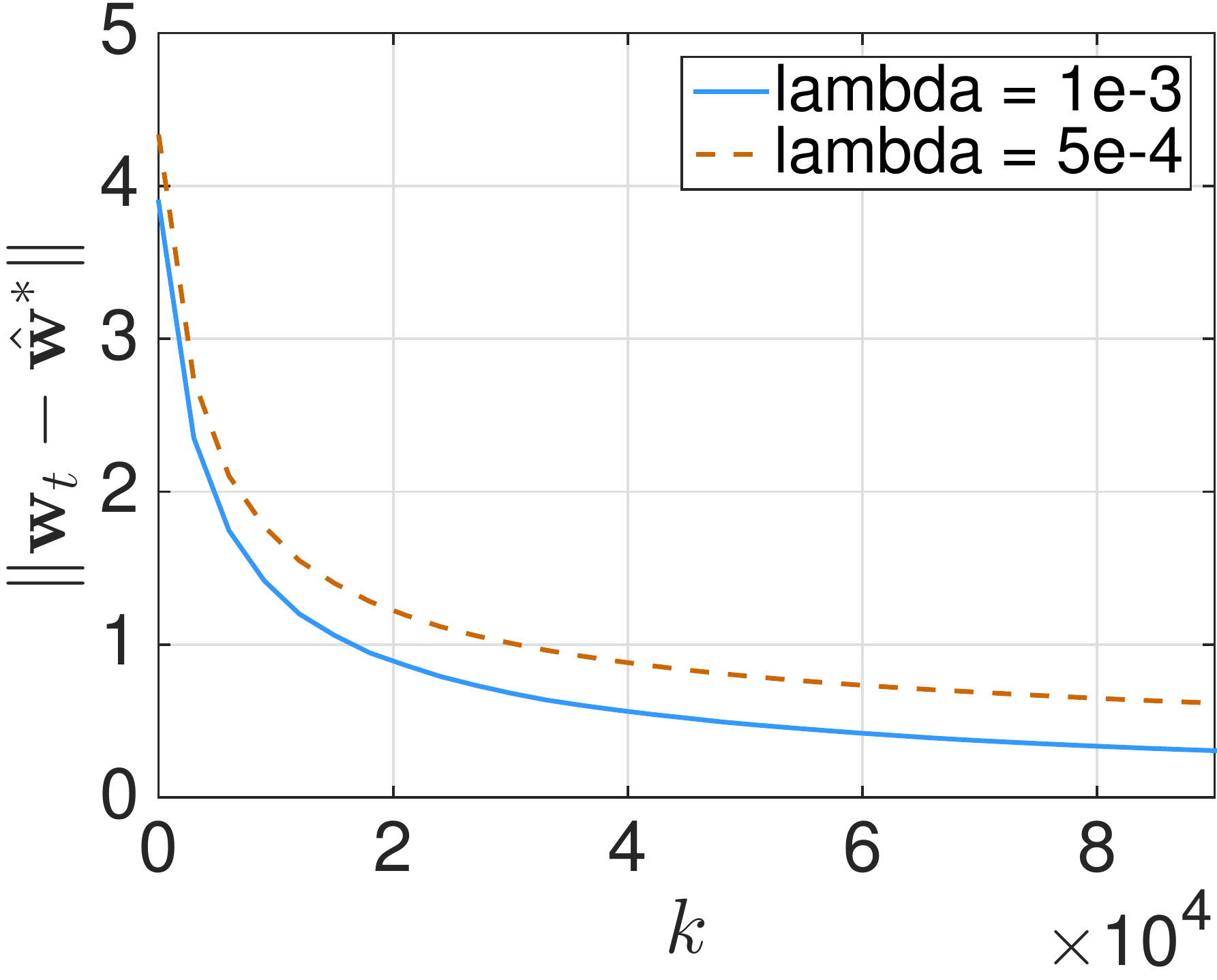}
}
\subfigure{
\includegraphics[width = 0.38\textwidth]{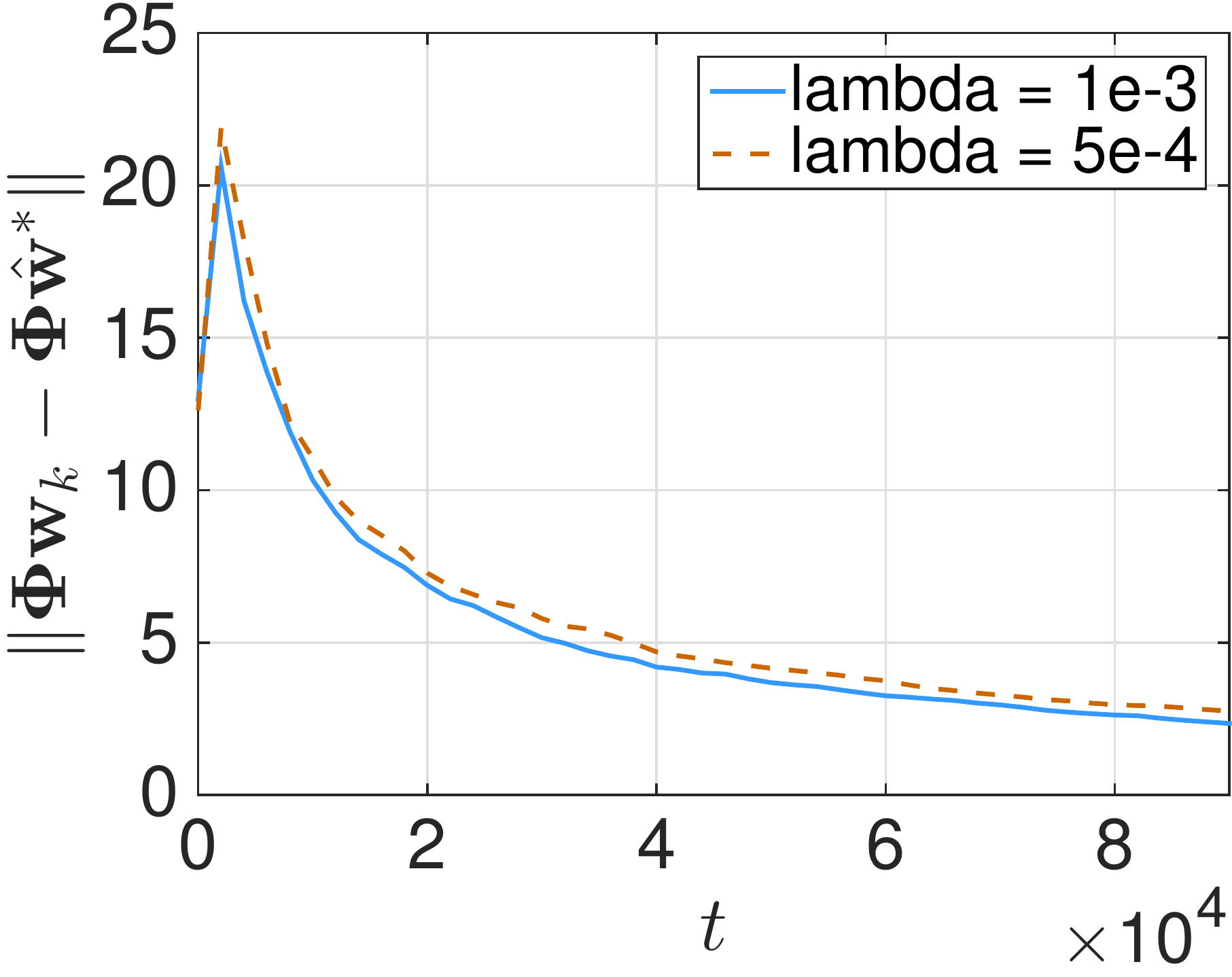}
}
\caption{Empirical convergence rate of the ASC-PG algorithm with the $\ell_1$-regularization term $\lambda\|\wb\|_1$ under Experiment 3 averaged over 100 runs, where $\wb_k$ denotes the solution at the $t$-th iteration.}
\label{fig:l1}
\end{figure}
\vspace{-10pt}

\section{Conclusion}

We develop a proximal gradient method for the penalized stochastic composition problem. The algorithm updates by interacting with a stochastic first-order oracle. Convergence rates are established under a variety of assumptions, which provide new rate benchmarks. Application of the ASC-PG to reinforcement learning leads to a new on-policy learning algorithm, which achieves faster convergence than the best known algorithms.
For future research, it remains open whether or under what circumstances the current $\cO(K^{-4/9})$ can be further improved. Another direction is to customize and adapt the algorithm and analysis to more specific problems arising from reinforcement learning and risk-averse optimization, in order to fully exploit the potential of the proposed method.
%
\newpage
\bibliographystyle{abbrvnat}
\bibliography{reference,SCGD}

\newpage
\begin{center}
\Large Supplemental Materials
\end{center}

\begin{lemma} \label{lem:bound:xdiff}
Under Assumption~\ref{ass:sg}, two subsequent iterates in Algorithm~\ref{alg} satisfy
\[
\|\x_k - \x_{k+1}\|^2 \leq \Theta(\alpha_k^2).
\]
\end{lemma}
\begin{proof}
From the definition of the proximal operation, we have
\begin{eqnarray*}
\x_{k+1} &=& \text{prox}_{\alpha_k R(\cdot)}(\x_k- \alpha_k\nabla g^\top_{w_k}(\x_k) \nabla f_{v_k} (\y_k)) 
\\ &=& \argmin_x~{1\over 2}\|\x- \x_k +\alpha_k\nabla g^\top_{w_k}(\x_k) \nabla f_{v_k} (\y_k)\|^2 + \alpha_kR(\x).
\end{eqnarray*}
The optimality condition suggests the following equality: 
\begin{align}
\x_{k+1} - \x_k = -\alpha_k(\nabla g^\top_{w_k}(\x_k) \nabla f_{v_k} (\y_k) + \s_{k+1})
\end{align}
where $\s_{k+1} \in \partial R(\x_{k+1})$ is some vector in the sub-differential set of $R(\cdot)$ at $\x_{k+1}$. Then apply the boundedness condition in Assumption~\ref{ass:sg} to yield
\begin{eqnarray*}
\|\x_{k+1} - \x_k\| &=& \alpha_k\|(\nabla g^\top_{w_k}(\x_k) \nabla f_{v_k} (\y_k) + \s_{k+1})\|
\\ 
&\leq & \alpha_k (\|\nabla g^\top_{w_k}(\x_k) \nabla f_{v_k} (\y_k)\| + \|\s_{k+1}\|)
\\
&\leq & \alpha_k (\|\nabla g^\top_{w_k}(\x_k)\| \|\nabla f_{v_k} (\y_k)\| + \|\s_{k+1}\|)
\\
&\stackrel{(\text{Assumption~\ref{ass:sg}})}{\leq} & \Theta(1)\alpha_k,
\end{eqnarray*}
which implies the claim.
\end{proof}

\begin{lemma} \label{lem:bound:gF}
Under Assumptions~\ref{ass:sg} and \ref{ass:lp}, we have
\begin{eqnarray*}
\|\nabla g^{\top}_{w}(\x) \nabla_{v} f(g(\x)) - \nabla g^\top_{w}(\x) \nabla_{v} f (\y))\| &\leq& \Theta(L_f\|\y - g(\x)\|).
\end{eqnarray*}
\end{lemma}
\begin{proof}
We have
\begin{eqnarray*}
\|\nabla g_w^\top(\x)\nabla f_v(g(\x)) - \nabla g_w^\top(\x) \nabla f_v (\y))\|
 &\leq &
\|\nabla g_w^\top (\x) \| \|\nabla f_v(g(\x)) - \nabla f_v(\y)\|
\\ & \stackrel{(\text{Assumption~\ref{ass:sg}})}{\leq}&
\Theta(1) \|\nabla f_v(g(\x)) - \nabla f_v(\y)\|
\\ & \stackrel{(\text{Assumption~\ref{ass:lp}})}{\leq}&
\Theta(L_f) \|\y - g(\x)\|.
\end{eqnarray*}
It completes the proof.
\end{proof}

\begin{lemma} \label{lem:gen-seq}
Given a positive sequence $\{w_k\}_{k=1}^{\infty}$ satisfying 
\begin{align}
w_{k+1} \leq (1-\beta_k + C_1 \beta_k^2)w_k + C_2 k^{-a} 
\end{align}
where $C_1\geq 0$, $C_2 \geq 0$, and $a \geq 0$. Choosing $\beta_k$ to be $\beta_k = C_3 k^{-b}$ where $b\in (0, 1]$ and $C_3> 2$,  the sequence can be bounded by
\[
w_{k} \leq C k^{-c} 
\]
where $C$ and $c$ are defined as 
\[
C := \max_{k \leq (C_1 C_3^2)^{1/b} + 1} w_k k^{c} + {C_2 \over C_3 -2} \quad \text{and} \quad c := a-b.
\]
In other words, we have
\[
w_k \leq \Theta(k^{-a+b}).
\]
\end{lemma}

\begin{proof}
We prove it by induction. First it is easy to verify that the claim holds for $k\leq (C_1C_3^2)^{1/b}$ from the definition for $C$. Next we prove from ``$k$'' to ``$k+1$'', that is, given $w_k \leq C k^{-c}$ for $k > (C_1C_3^2)^{1/b}$, we need to prove $w_{k+1} \leq C(k+1)^{-c}$.
\begin{eqnarray}
\nonumber
w_{k+1} & \leq & (1-\beta_k + C_1\beta_k^2) w_k + C_2 k^{-a}
\\ \nonumber & \leq & 
(1-C_3k^{-b} + C_1C_3^2 k^{-2b}) C k^{-c} + C_2 k^{-a}
\\ & = & 
C k^{-c} - CC_3k^{-b-c} + CC_1C_3^2 k^{-2b -c} + C_2 k^{-a}.
\label{eq:lemma:basicseq}
\end{eqnarray}
To prove that \eqref{eq:lemma:basicseq} is bounded by $C(k+1)^{-c}$, it suffices to show that
\[
\Delta := (k+1)^{-c} - k^{-c} + C_3k^{-b-c} - C_1C_3^2 k^{-2b-c} > 0\quad \text{and}\quad
C \geq \frac{C_2 k^{-a}}{\Delta}.
\]
From the convexity of function $h(t)=t^{-c}$, we have the inequality $(k+1)^{-c} - k^{-c} \geq (-c)k^{-c-1}$. Therefore we obtain 
\begin{eqnarray*}
\Delta & \geq & -c k^{-c-1} + C_3 k^{-b-c} - C_1 C_3^2 k^{-2b-c}
\\ &\stackrel{(b\leq 1,~k>(C_1C_3^2)^{1/b})}{\geq} & (C_3-2)(k^{-b-c}) 
\\ & \stackrel{(C_3 > 2)}{>} & 0.
\end{eqnarray*}
To verify the second one, we have
\[
\frac{C_2 k^{-a}}{\Delta}\leq \frac{C_2}{C_3-2} k^{-a+b+c} \stackrel{(c=a+b)}{=} \frac{C_2}{C_3-2} \leq C
\]
where the last inequality is due to the definition of $C$. It completes the proof.
\end{proof}

\begin{lemma} \label{lem:y}
Choose $\beta_k$ to be $\beta_k = C_b k^{-b}$ where $C_b>2$, $b\in (0,1]$, and $\alpha_k = C_a k^{-a}$. Under Assumptions~\ref{ass:unb} and \ref{ass:var}, we have
\begin{align}
\E(\|\y_{k} - g(\x_{k})\|^2) \leq L_g\Theta(k^{-4a+4b}) + \Theta(k^{-b}).
 \end{align}
\end{lemma}

\begin{proof}
Denote by $m_{k+1}$
\[
m_{k+1} := \sum_{t=0}^k \xi_t^{(k)} \|\x_{k+1} - \z_{t+1}\|^2
\] 
and $n_{k+1}$
\[
n_{k+1} := \left\|\sum_{t=0}^k \xi_{t}^{(k)} (g_{w_{t+1}}(\z_{t+1}) - g(\z_{t+1})) \right\|
\] 
for short.

From Lemma~10 in \citep{wang2014stochastic}, we have
\begin{eqnarray}
\|\y_{k} - g(\x_{k})\|^2 \leq \left({L_g\over 2} m_{k} + n_{k}\right)^2 \leq L_g m_{k}^2 + 2n_{k}^2.
\label{eq:proof:lem:0}
\end{eqnarray}

From Lemma~11 in \citep{wang2014stochastic}, $m_{k+1}$ can be bounded by
\begin{eqnarray}
\label{eq:proof:lem:y:1}
m_{k+1} &\leq& (1-\beta_k) m_k + \beta_k q_k + {2 \over \beta_k} \|\x_k - \x_{k+1}\|^2
\end{eqnarray}
where $q_k$ is bounded by 
\begin{eqnarray*}
q_{k+1} & \leq & (1-\beta_k) q_k + {4\over \beta_k} \|\x_{k+1} - \x_k\|^2
\\ & \stackrel{\text{(Lemma~\ref{lem:bound:xdiff})}}{\leq} & 
(1-\beta_k) q_k + {\Theta(1) \alpha_k^2\over \beta_k}
\\ & \leq &
(1-\beta_k) q_k + \Theta(k^{-2a+b}).
\end{eqnarray*}
Use Lemma~\ref{lem:gen-seq} and obtain the following decay rate
\[
q_{k} \leq \Theta(k^{-2a+2b}).
\]
Together with \eqref{eq:proof:lem:y:1}, we have
\begin{eqnarray*}
m_{k+1} &\leq& (1-\beta_k) m_k + \beta_k q_k + {2 \over \beta_k} \|\x_k - \x_{k+1}\|^2
\\ & \leq & (1-\beta_k) m_k + \Theta(k^{-2a+b}) + \Theta(k^{-2a+b})
\\ & \leq & (1-\beta_k) m_k + \Theta(k^{-2a+b}),
\end{eqnarray*}
which leads to
\begin{equation}
m_{k} \leq \Theta(k^{-2a + 2b})\quad \text{and} \quad m_{k}^2 \leq \Theta(k^{-4a+4b}).
\label{eq:proof:m:bound}
\end{equation}
by using Lemma~\ref{lem:gen-seq} again. Then we estimate the upper bound for $\E(n_k^2)$. From Lemma~11 in \citep{wang2014stochastic}, we know $\E(n_k^2)$ is bounded by
\[
\E(n_{k+1}^2) \leq (1-\beta_k)^2 \E(\|n_k\|^2) + \beta_k^2 \sigma_g^2 = (1-2\beta_k + \beta_k^2) \E(\|n_k\|^2) + \beta_k^2 \sigma_g^2.
\]
By using Lemma~\ref{lem:gen-seq} again, we have 
\begin{equation}
\E(n_{k}^2) \leq \Theta(k^{-b}).
\label{eq:proof:lem:n}
\end{equation}

Now we are ready to estimate the upper bound of $\|\y_{k+1} - g(\x_{k+1})\|^2$ by following \eqref{eq:proof:lem:0}
\begin{eqnarray*}
\E(\|\y_{k} - g(\x_{k})\|^2) & \leq & {L_g}\E(m_{k}^2) + 2\E(n^2_{k})
\\ & \stackrel{\eqref{eq:proof:m:bound} + \eqref{eq:proof:lem:n}}{\leq} &
L_g\Theta(k^{-4a+4b}) + \Theta(k^{-b}).
\end{eqnarray*}
It completes the proof.
\end{proof}

{\bf \noindent Proof to Theorem~\ref{thm:nonconvex}}
\begin{proof}
From the Lipschitzian condition in Assumption~\ref{ass:lp}, we have
\begin{eqnarray}
\nonumber
& &F(\x_{k+1}) - F(\x_k) 
\\ \nonumber & \leq & 
\la \nabla F(\x_k),~ \x_{k+1} - \x_k\ra + {L_F \over 2} \|\x_{k+1} - \x_k\|^2
\\  \nonumber & \stackrel{\text{(Lemma~\ref{lem:gen-seq})}}{\leq} & 
-\alpha_k \la \nabla F(\x_k),~ \Gg_{w_k}^\top(\x_k) \Gf_{v_k}(\y_{k}) \ra + \Theta(\alpha_k^2)
\\ & = & 
-\alpha_k \|\nabla F(\x_k)\|^2 + \alpha_k \underbrace{\la \nabla F(\x_k),~ \nabla F(\x_k) - \Gg_{w_k}^\top(\x_k) \Gf_{v_k}(\y_{k}) \ra}_{=:T} \nonumber\\
&&+ \Theta(\alpha_k^2)
\label{eq:proof:nonconvex:1}
\end{eqnarray}
Next we estimate the upper bound for $\E(T)$:
\begin{eqnarray*}
\E(T) & = & \E(\la \nabla F(\x_k),~ \nabla F(\x_k) - \Gg_{w_k}^\top(\x_k) \Gf_{v_k}(g(\x_k)) \ra) 
\\ && 
+ \E(\la \nabla F(\x_k), \Gg_{w_k}^\top(\x_k) \Gf_{v_k}(g(\x_k)) - \Gg_{w_k}^\top(\x_k) \Gf_{v_k}(\y_{k}) \ra )  
\\ & \stackrel{\text{(Assumption \ref{ass:unb})}}{=} &
\E(\la \nabla F(\x_k), \Gg_{w_k}^\top(\x_k) \Gf_{v_k}(g(\x_k)) - \Gg_{w_k}^\top(\x_k) \Gf_{v_k}(\y_{k}))\ra )  
\\ & \leq & 
{1 \over 2} \E(\|\nabla F(\x_k)\|^2) + {1 \over 2} \E(\|\Gg_{w_k}^\top(\x_k) \Gf_{v_k}(g(\x_k)) - \Gg_{w_k}^\top(\x_k) \Gf_{v_k}(\y_{k})\|^2)
\\ & \stackrel{\text{(Lemma~\ref{lem:bound:gF})}}{\leq} &
{1 \over 2}\E(\|\nabla F(\x_k)\|^2) + \Theta(L_f^2) \E(\|\y_{k} - g(\x_k)\|^2).
\end{eqnarray*}
Take expectation on both sides of \eqref{eq:proof:nonconvex:1} and substitute $\E(T)$ by its upper bound:
\begin{eqnarray*}
&&{\alpha_k\over 2} \|\nabla F(\x_k)\|^2 
\\ &\leq& 
\E(F(\x_k)) -  \E(F(\x_{k+1})) + \Theta(L_f^2\alpha_k) \E(\|\y_{k} - g(\x_k)\|^2) + \Theta(\alpha_k^2) 
\\ & \stackrel{\text{(Lemma~\ref{lem:y})}}{\leq} &
\E(F(\x_k)) -  \E(F(\x_{k+1})) + L_g\Theta(L_f^2\alpha_k) \Theta(k^{-4a+4b}) + \Theta(L_f^2\alpha_k k^{-b}) + \Theta(\alpha_k^2)
\\ & \leq &
\E(F(\x_k)) -  \E(F(\x_{k+1})) + L_f^2L_g\Theta(k^{-5a+4b}) + L_f^2\Theta(k^{-a-b}) + \Theta(k^{-2a}) 
\end{eqnarray*}
which suggests that
\begin{eqnarray}
\nonumber
&& \E(\|\nabla F(\x_k)\|^2) \\& \leq & 
2\alpha_k^{-1}\E(F(\x_k)) - 2\alpha_k^{-1} \E(F(\x_{k+1})) + L_f^2L_g\Theta(k^{-4a+4b}) + L_f^2\Theta(k^{-b}) + \Theta(k^{-a}) 
\nonumber \\ & \leq & 
{2k^a}\E(F(\x_k)) - {2k^a} \E(F(\x_{k+1})) + L_f^2 L_g\Theta(k^{-4a+4b}) + L_f^2 \Theta(k^{-b}) + \Theta(k^{-a})
\label{eq:proof:nonconvex:dF}
\end{eqnarray}


Summarize Eq.~\eqref{eq:proof:nonconvex:dF} from $k=1$ to $K$ and obtain
\begin{eqnarray*}
\frac{\sum_{k=1}^K \E(\|\nabla F(\x_k)\|^2)}{K} & \leq & 2K^{-1}\alpha_1^{-1}F(\x_1) + {K^{-1}}\sum_{k=2}^{K} ((k+1)^a - k^a) \E(F(\x_{k}))  
\\  && + K^{-1}\sum_{k=1}^K L_f^2 L_g\Theta(k^{-4a+4b}) + {K^{-1}}L_f^2 \sum_{k=1}^K \Theta(k^{-b}) + K^{-1}\sum_{k=1}^K \Theta(k^{-a})
\\ & \leq & 
2K^{-1}F(\x_0) + {K^{-1}}\sum_{k=2}^{K} ak^{a-1} \E(F(\x_{k}))  
\\  && + K^{-1}\sum_{k=1}^K L_f^2 L_g\Theta(k^{-4a+4b}) + {K^{-1}}L_f^2 \sum_{k=1}^K \Theta(k^{-b}) + K^{-1}\sum_{k=1}^K \Theta(k^{-a})
\\ & \leq &
O(K^{a-1} + L_f^2 L_gK^{4b-4a}\I_{4a-4b=1}^{\log K} + L_f^2 K^{-b}+ K^{-a}),
\end{eqnarray*}
where the second inequality uses the fact that $h(t) = t^a$ is a concave function suggesting $(k+1)^a \leq k^a + ak^{a-1}$, and the last inequality uses the condition $\E(F(\x_k)) \leq \Theta(1)$.

The optimal $a^* = 5/9$ and the optimal $b^* = 4/9$, which leads to the convergence rate $O(K^{-4/9})$.
\end{proof}

{\bf \noindent Proof to Theorem~\ref{thm:strconvex}}
\begin{proof}
Following the line of the proof to Lemma~\ref{lem:bound:xdiff}, we have
\begin{align}
\x_{k+1} - \x_k = -\alpha_k(\nabla g^\top_{w_k}(\x_k) \nabla f_{v_k} (\y_k) + \s_{k+1})
\label{eq:proof:thm2:opt}
\end{align}
where $\s_{k+1} \in \partial R(\x_{k+1})$ is some vector in the sub-differential set of $R(\cdot)$ at $\x_{k+1}$. Then we consider $\|\x_{k+1} - \P_{X^*}(\x_{k+1})\|^2$:
\begin{eqnarray}
\nonumber
&& \|\x_{k+1} - \P_{X^*}(\x_{k+1})\|^2 
\\ &\leq& \|\x_{k+1} - \x_{k} + \x_k - \P_{X^*}(\x_k)\|^2
\nonumber
\\ &=&
\|\x_k - \P_{X^*}(\x_k)\|^2 - \|\x_{k+1} - \x_k\|^2 + 2\la \x_{k+1} - \x_k, \x_{k+1} - \P_{X^*}(\x_k) \ra
\nonumber
\\ &\stackrel{\eqref{eq:proof:thm2:opt}}{=}&
\|\x_k - \P_{X^*}(\x_k)\|^2 - \|\x_{k+1} - \x_k\|^2 - 2\alpha_k \la \nabla g^\top_{w_k}(\x_k) \nabla f_{v_k} (\y_k) + \s_{k+1},~\x_{k+1} - \P_{X^*}(\x_k) \ra
\nonumber
\\ &= &
\|\x_k - \P_{X^*}(\x_k)\|^2 - \|\x_{k+1} - \x_k\|^2 + 2\alpha_k \la \nabla g^\top_{w_k}(\x_k) \nabla f_{v_k} (\y_k),  \P_{X^*}(\x_k) - \x_{k+1} \ra 
\nonumber
\\ &&
+ 2\alpha_k \la \s_{k+1},~\P_{X^*}(\x_k) - \x_{k+1}\ra
\nonumber
\\ &\leq&
\|\x_k - \P_{X^*}(\x_k)\|^2 - \|\x_{k+1} - \x_k\|^2 + 2\alpha_k \la \nabla g^\top_{w_k}(\x_k) \nabla f_{v_k} (\y_k),  \P_{X^*}(\x_k) - \x_{k+1} \ra 
\nonumber \\ &  &
 + 2\alpha_k (R(\P_{X^*}(\x_k)) - R(\x_{k+1}))\quad\quad {(\text{due to the convexity of $R(\cdot)$})}
\nonumber
\\ &\leq &
\|\x_k - \P_{X^*}(\x_k)\|^2 - \|\x_{k+1} - \x_k\|^2 + 2\alpha_k \underbrace{\la \nabla F(\x_k),  \P_{X^*}(\x_k) - \x_{k+1} \ra}_{T_1}
\nonumber
\\ &&
+ 2\alpha_k \underbrace{\la \nabla g^\top_{w_k}(\x_k) \nabla f_{v_k} (\y_k) - \nabla F(\x_k),~\P_{X^*}(\x_k) - \x_{k+1}\ra}_{T_2}
\nonumber
\\ &&
+ 2\alpha_k (R(\P_{X^*}(\x_k)) - R(\x_{k+1}))
\label{eq:proof:thm2:2}
\end{eqnarray}
where the second equality follows from $\|a+b\|^2 = \|b\|^2 - \|a\|^2 + 2\la a,~a+b \ra$ with $a=\x_{k+1} - \x_k$ and $b = \x_k - \P_{X^*}(\x_k)$. We next estimate the upper bound for $T_1$ and $T_2$ respectively:
\begin{eqnarray*}
T_1 & = & \la \nabla F(\x_k),~\x_k - \x_{k+1} \ra + \la \nabla F(\x_k),~-\x_k + \P_{X^*}(\x_k)\ra
\\ & {\leq} &
\underbrace{F(\x_k) - F(\x_{k+1}) + {L_F\over 2} \|\x_{k+1} - \x_k\|^2}_{{\text{due to Assumption}~\ref{ass:lp}}} + \underbrace{F(\P_{X^*}(\x_k)) - F(\x_{k})}_{\text{due to the convexity of $F(\cdot)$}}
\\ & = &
F(\P_{X^*}(\x_k)) - F(\x_{k+1}) + {L_F\over 2} \|\x_{k+1} - \x_k\|^2
\\ &\leq& 
F(\P_{X^*}(\x_k)) - F(\x_{k+1}) + \Theta(\alpha_k^2),
\end{eqnarray*}
where the last inequality uses Lemma~\ref{lem:bound:xdiff}.

\begin{eqnarray*}
T_2 & = & \la \nabla F(\x_k)- \nabla g^\top_{w_k}(\x_k) \nabla f_{v_k} (\y_k),  \x_k - \P_{X^*}(\x_k) \ra + \la \nabla F(\x_k) -  \nabla g^\top_{w_k}(\x_k) \nabla f_{v_k} (\y_k),~ \x_{k+1} - \x_k \ra
\\ & \leq & 
\underbrace{ \la \nabla F(\x_k) - \nabla g^\top_{w_k}(\x_k) \nabla f_{v_k}(g(\x_k)),~\x_k - \P_{X^*}(\x_k) \ra }_{T_{2,1}}
 \\ && +  \underbrace{\la \nabla g^\top_{w_k}(\x_k) \nabla f_{v_k}(g(\x_k)) - \nabla g^\top_{w_k}(\x_k) \nabla f_{v_k} (\y_k),~\x_k - \P_{X^*}(\x_k) \ra}_{T_{2,2}}
 \\ && + {\alpha_k \over 2}\underbrace{\|\nabla F(\x_k) -  \nabla g^\top_{w_k}(\x_k) \nabla f_{v_k} (\y_k)\|^2}_{T_{2,3}} + {1 \over 2 \alpha_k}\|\x_k - \x_{k+1}\|^2
\end{eqnarray*}
where the last line is due to the inequality $\la a,b \ra \leq {1\over 2\alpha_k}\|a\|^2 + {\alpha_k \over 2}\|b\|^2$. For $T_{2,1}$, we have $\E(T_{2,1}) =0 $ due to Assumption~\ref{ass:unb}. For $T_{2,2}$, we have 
\begin{eqnarray*}
T_{2,2} & {\leq} &  {\alpha_k \over 2 \phi_k}\|\nabla g^\top_{w_k}(\x_k) \nabla f_{v_k}(g(\x_k)) - \nabla g^\top_{w_k}(\x_k) \nabla f_{v_k} (\y_k)\|^2 + {\phi_k \over 2\alpha_k} \|\x_k - \P_{X^*}(\x_k)\|^2
\\ & \stackrel{\text{(Lemma~\ref{lem:bound:gF})}}{\leq} & \Theta\left(L_f^2{\alpha_k\over \phi_k}\right) \|\y_k - g(\x_k)\|^2 + {\phi_k \over 2\alpha_k} \|\x_k - \x_{k+1}\|^2.
\end{eqnarray*}
$T_{2,3}$ can be bounded by a constant
\[
T_{2,3} \leq 2\|\nabla F(\x_k)\|^2 + 2\|\nabla g^{\top}_{w_k} \nabla f_{v_k}(\y_k)\|^2 \stackrel{(\text{Assumption~\ref{ass:sg}})}{\leq} \Theta(1).
\]

Take expectation on $T_2$ and put all pieces into it:
\begin{eqnarray*}
\E(T_2) & \leq & \Theta\left(L_f^2{\alpha_k \over \phi_k}\right) \|\y_k - g(\x_k)\|^2 + {1 \over 2\alpha_k} (\phi_k\|\x_k - \P_{X^*}(\x_k)\|^2 + \|\x_k - \x_{k+1}\|^2) + \Theta(\alpha_k).
\end{eqnarray*}

Taking expectation on both sides of \eqref{eq:proof:thm2:2} and plugging the upper bounds of $T_1$ and $T_2$ into it, we obtain
\begin{align*}
&2\alpha_k (\E(H(\x_{k+1})) - H^*) + \E(\|\x_{k+1} - \P_{X^*}(\x_{k+1})\|^2) 
\\ &
\quad \leq (1+\phi_k) \E(\|\x_k - \P_{X^*}(\x_k)\|^2) + \Theta(\alpha_k^3) + \Theta(L_f^2\alpha_k^2/\phi_k) \E(\|\y_k - g(\x_k)\|^2) + \Theta(\alpha_k^2).
\end{align*}
Using the optimally strong convexity in \eqref{eq:osc}, we have
\begin{align*}
&(1+2\lambda \alpha_k) \E(\|\x_{k+1} - \P_{X^*}(\x_{k+1})\|^2) 
\\
&\quad \quad \leq (1+\phi_k) \E(\|\x_k - \P_{X^*}(\x_k)\|^2) + \Theta(\alpha_k^3) + \Theta(L_f^2\alpha_k^2/\phi_k) \E(\|\y_k - g(\x_k)\|^2) + \Theta(\alpha_k^2).
\end{align*}
It follows by dividing $1+2\lambda \alpha_k$ on both sides
\begin{align*}
& \E(\|\x_{k+1} - \P_{X^*}(\x_{k+1})\|^2) \\
& \quad \quad
 \leq \frac{1+\phi_k}{1+2\lambda \alpha_k} \E(\|\x_k - \P_{X^*}(\x_k)\|^2) + \Theta(\alpha_k^3) + \Theta(L_f^2\alpha_k^2/\phi_k) \E(\|\y_k - g(\x_k)\|^2) + \Theta(\alpha_k^2).
\end{align*}

Choosing $\phi_k = \lambda \alpha_k - 2\lambda^2 \alpha_k^2 \geq {0.5\lambda \alpha_k}$ yields 
\begin{eqnarray*}
&& \E(\|\x_{k+1} - \P_{X^*}(\x_{k+1})\|^2)
\\ & \leq & (1-\lambda \alpha_k) \E(\|\x_{k} - \P_{X^*}(\x_k)\|^2) + \Theta(\alpha_k^2) + {\Theta(L_f^2\alpha_k)\over \lambda} \E(\| g(\x_k) - \y_k\|^2)
\\ & \leq & 
(1-\lambda \alpha_k) \E(\|\x_{k} - \P_{X^*}(\x_k)\|^2) + \Theta(k^{-2a}) + \Theta(L_gL_f^2 k^{-5a+4b} + L_f^2k^{-a-b}). 
\end{eqnarray*}
Apply Lemma~\ref{lem:gen-seq} to obtain the first claim in \eqref{eq:thm:osc:1}
\[
\E(\|\x_{k} - \P_{X^*}(\x_k)\|^2) \leq O\left(k^{-a} + L_f^2L_gk^{-4a+4b} + L_f^2k^{-b}\right).
\]
The followed specification of $a$ and $b$ can easily verified.
\end{proof}

\end{document}